\documentclass[11pt]{article}
\usepackage{geometry}                		

\usepackage{amssymb}
\usepackage{amsmath}

\usepackage{palatino}

\usepackage{graphicx}
\usepackage[english]{babel}
\usepackage{amscd}
\usepackage{amsthm}

\usepackage{color}
\usepackage{enumerate}

\newtheorem{theorem}{Theorem}[section]
\newtheorem{lemma}[theorem]{Lemma}

\newtheorem{definition}{Definition}
\newtheorem{conjecture}[theorem]{Conjecture}

\newcommand{\Poi}{\textit{Poi}}
\newcommand{\Bin}{\textit{Bin}}
\newcommand{\Prb}{\mathbb{P}}
\newcommand{\Exp}{\mathbb{E}}

\title{On the monotonicity of tail probabilities}

\author{
Robbert Fokkink\thanks{Institute of Applied Mathematics, 
Delft University of Technology, 
Mourikbroekmanweg 6
2628 XE Delft, The Netherlands,
e-mail: r.j.fokkink@tudelft.nl} 
\and
Symeon Papavassiliou\thanks{School of Electrical and Computer Engineering, 
National Technical University of Athens, 
Zografou, Greece, 15780, e-mail: papavass@mail.ntua.gr}
\and
Christos Pelekis\thanks{School of Electrical and Computer Engineering, 
National Technical University of Athens, 
Zografou, Greece, 15780, e-mail: pelekis.chr@gmail.com}
}

\begin{document}
	\maketitle
	
\begin{abstract}  
Let $S$ and $X$ be independent random 
variables, assuming values in the set of non-negative integers, and suppose further 
that both $\Exp(S)$ and $\Exp(X)$ are integers satisfying $\Exp(S)\ge \Exp(X)$.  
We establish a sufficient condition for the tail  probability $\Prb(S\ge \Exp(S))$ 
to be larger than $\Prb(S+X\ge \Exp(S+X))$. We  also apply this result to sums of independent binomial and Poisson random variables.  
\end{abstract}

	

\section{Main result}  

Here and later, given a positive integer $m\ge 1$, we denote by $[m]$ the set $\{1,\ldots,m\}$ and by 
$\Bin(m,p)$ a binomial random variable of parameters $m$ and $p\in (0,1)$. 
Moreover,  $\Poi(\lambda)$ denotes a 
Poisson random variable of mean $\lambda$. The notation $Y\sim Z$ indicates that the random variables $Y,Z$ have the same distribution.

Let $\{X_i\}_{i\ge 1}$ be a sequence of \emph{independent} random variables, 
assuming values in the set of non-negative integers, and 
for $k\ge 1$, let $S_k$ 
denote the  sum 
$S_k = X_1+ \cdots + X_k$. 
We shall be concerned with  monotonicity properties for tail probabilities of the form $\{\Prb(S_k \ge \Exp(S_k))\}_{k\ge 1}$. 

There is a considerable amount of literature 
on tail monotonicity of particular distributions. 
It is shown in~\cite{Chaundy_Bullard} that for every fixed positive integer $n\ge 1$, it holds 
\begin{equation}\label{CB_ineq}
\Prb\left(\Bin(nk, \frac{1}{n})\ge k\right) \ge 
\Prb\left(\Bin(n(k+1), \frac{1}{n})\ge k+1\right), \, \text{ for all  } \, k\ge 1 \, .
\end{equation}

An inequality similar to~\eqref{CB_ineq} is also valid for Poisson random variables.  
It is shown in~\cite{Teicher} that 
\begin{equation}\label{Teicher_ineq}
\Prb(\Poi(k)\ge k) \ge \Prb(\Poi(k+1)\ge k+1),  \, \text{ for all  } \, k\ge 1 \, .
\end{equation}
A slightly different proof of~\eqref{Teicher_ineq} can be found in~\cite[Lemma~1]{Adell}.  
Further inequalities can be found in~\cite{Anderson_Samuels, Kane, Pinelis}, among others.

Let us remark that  inequalities~\eqref{CB_ineq} and~\eqref{Teicher_ineq} 
concern the monotonicity of tail probabilities of the form 
$\Prb(S_k \ge \Exp(S_k))$, where $S_k$ is a sum of  independent random whose means are  equal to $1$. Both inequalities~\eqref{CB_ineq} and~\eqref{Teicher_ineq} have been extended to the case of integer means (see~\cite[Theorem~2.1]{Jogdeo_Samuels} and~\cite[Theorem~2.3]{Kane}).

This note aims at demonstrating that some special cases as well as some extensions of comparisons between tail probabilities for sums of independent random variables, whose means are all integers,  can be derived from a 
comparison between 
$\Prb(S+X\ge s+\Exp(X))$ and $\Prb(S\ge s)$, where 
$S$ and $X$ are independent 
random variables 
satisfying certain 
``skewness" conditions, which we now describe formally.  

\begin{definition}[Right-skewness]
\label{Rskew}
Let $S$ take values in the set of non-negative integers. Assume further that $S$ is unimodal with mode $s$. The $S$ is called \emph{right-skewed} if   
\[
\Prb(S = s -i) \le \Prb(S = s + i-1)\, ,
\text{ for all } i\in [s] \, . 
\]
\end{definition}
 
\begin{definition}[Left-loadedness]
\label{Scond} 
Let $X$ be a random variable, taking values in the set of non-negative integers, and suppose  further that $m:=\Exp(X)$ is an \emph{integer}. For $i\in [m]$, set 
$\alpha_i := \Prb(X\le m-i)-\Prb(X\ge m+i)$. 
Then $X$ is said to be \emph{left-loaded} if either of the following two conditions holds true:
\begin{enumerate} 
\item \textbf{Condition $(L_1)$}: The sequence $\{\alpha_i\}_{i=1}^m$ changes sign once, i.e., there exists $\ell\in [m]$ such that $\alpha_i\ge 0$, for $i\le \ell$, and $\alpha_i \le 0$, for $i>\ell$.  
\item \textbf{Condition $(L_2)$}: It holds $\sum_{i=1}^{k}\alpha_i \ge 0$, for all $k\in [m]$. 
\end{enumerate}
\end{definition}

Our main result   
reads as follows.

\begin{theorem}\label{gen_thm}
Let $s\ge m$ be two positive integers. Suppose that $S$ and $X$  are  \emph{independent} random variables, assuming values in the set of
non-negative integers, that satisfy 
the following conditions:
\begin{enumerate}
    \item $S$ is right-skewed with mode $s$.  
    \item $X$ is left-loaded with mean $m$. 
\end{enumerate}
Then $\Prb(S \ge s) \ge \Prb(S+X \ge s+ m)\,$.
\end{theorem}

We prove Theorem~\ref{gen_thm} in Section~\ref{sec:2}. In Section~\ref{sec:4} we apply Theorem~\ref{gen_thm} to sums of independent Poisson random variables, 
and deduce a refinement of a result due to 
Kane~\cite[Theorem~2.3]{Kane}. 
Our paper ends with Section~\ref{sec:last}, where 
we briefly sketch how Theorem~\ref{gen_thm} applies to sums of independent binomial random variables, and state  two  conjectures.

\section{Proof of Theorem~\ref{gen_thm}}\label{sec:2}

We begin with an observation.

\begin{lemma}\label{lem:1}
Let $X$ be random variable, assuming non-negative integer values, such that  $m:=\Exp(X)$ is an \emph{integer}. Then 
\[
\sum_{i=1}^{m} \big(\Prb(X\le m-i) - \Prb(X\ge m+i)\big) = \sum_{i\ge m+1} \Prb(X\ge m+i) \, .
\]
In particular,  $\sum_{i=1}^{m} \big(\Prb(X\le m-i) - \Prb(X\ge m+i)\big) \ge 0$.
\end{lemma}
\begin{proof}
Notice that  
\[
m =   \sum_{i=1}^m\Prb(X\ge i) + \sum_{i=m+1}^{2m}\Prb(X\ge i) + \sum_{i\ge 2m+1}\Prb(X\ge i)  \, ,
\]
which, upon transferring the first two sums on the right to the other side, is equivalent to 
\[
\sum_{i=1}^{m} \left(\Prb(X\le m-i) - \Prb(X\ge m+i)\right) = \sum_{i\ge m+1} \Prb(X\ge m+i) \, . 
\]
The result follows. 
\end{proof}

\begin{proof}[Proof of Theorem~\ref{gen_thm}]
If we condition on $S$ we have     
\begin{eqnarray*}
\Prb(S+X\ge s+m) &=& \sum_{i\ge 0} \Prb(X\ge s+m-i) \cdot \Prb(S =i) \\ 
&=&\Prb(S \ge s+m) + \sum_{i=0}^{s+m -1}\Prb(X\ge s+m-i) \cdot\Prb(S=i)
\end{eqnarray*}
Hence $\Prb(S+X\ge s+m )\le \Prb(S\ge s)$ will follow from
\begin{equation*}
\sum_{i=0}^{s+m-1} \Prb(X\ge s+m-i)\cdot\Prb(S =i) \,\le\, \sum_{i=1}^{m}\Prb( S=  s+i -1) \, ,
\end{equation*}
which is equivalent to
\begin{equation}\label{enough:01}
\sum_{i= 1}^{s} \Prb(S = s -i)\cdot \Prb(X\ge m+i)\, \le\, \sum_{i=1}^{m} \Prb(S=s +i-1) \cdot \Prb(X \le m -i)  \, .
\end{equation} 
Let $L$ and $R$ denote the left hand side 
and the right hand side of~\eqref{enough:01}.
Since $S$ is unimodal with mode equal to $s$, and $s\ge m$, we 
can estimate $L$ as follows:
\begin{equation*}
L 
\,\le\, \underbrace{\sum_{i= 1}^{m } \Prb(S = s -i)\cdot \Prb(X \ge m +i)}_\textrm{$L_1$}
\, + \,  
\underbrace{\Prb(S = s - m -1) \cdot \sum_{i=m +1}^{s}\Prb(X \ge m+i)}_\textrm{$L_2$} \, ,
\end{equation*}
with the convention that $L_2=0$ when $s=m$.
Now, since $S$ is right-skewed, we have    
\begin{equation}\label{ar:111}
    L_1 \le \sum_{i= 1}^{m} \Prb(S = s + i-1)\cdot \Prb(X\ge m+i) =: R_1 \, .
\end{equation}
Using again the right-skewness of $S$ and Lemma~\ref{lem:1}, we have   
\begin{equation}\label{ar:101}
L_2\le \Prb(S = s+m ) \cdot \left(
\sum_{i=1}^{m}(\Prb(X\le m -i) - \Prb(X\ge m+ i))\right) =: R_2 \, .
\end{equation}
It follows from~\eqref{enough:01},~\eqref{ar:111} and~\eqref{ar:101} that it is enough 
to show that $R_1 + R_2 \le R$ or, equivalently,  
\begin{equation}\label{eq:final}
\sum_{i=1}^m \big(\Prb(S=s+i-1) - \Prb(S=s+m) \big) \cdot\big( \Prb(X\le m-i)  -  \Prb(X\ge m+i) \big)\, \ge\, 0 \, ,
\end{equation}
For $i\in [m]$, let 
$\Delta_i := \Prb(S=s+i-1) - \Prb(S=s+m)$
and $\alpha_i := \Prb(X\le m-i)-\Prb(X\ge m+i)$, and note that~\eqref{eq:final} is equivalent to 
\begin{equation}\label{positive}
    \sum_{i=1}^{m} \Delta_i \cdot \alpha_i \ge 0 \, . 
\end{equation}
The 
unimodality of $S$ implies that $\Delta_1\ge\cdots\ge \Delta_m\ge 0$. 
We distinguish two cases.

Suppose first that $X$ satisfies   Condition $(L_1)$. 
Let $\ell\in [m]$ be such that $\alpha_i\ge 0$, for $i\le \ell$, and $\alpha_i \le 0$, for $i>\ell$.
Then, since $\{\Delta_i\}_{i\in [m]}$ is non-increasing, it holds 
\[
\sum_{i=1}^{m} \Delta_i \cdot \alpha_i \ge \Delta_{\ell}\sum_{i=1}^{\ell} \alpha_i + 
\Delta_{\ell} \sum_{i=\ell+1}^{m} \alpha_i = \Delta_{\ell} \sum_{i\in [m]} \alpha_i\ge 0 \, ,
\]
where the last estimate follows from the 
second statement in Lemma~\ref{lem:1}. 
Hence~\eqref{positive} holds. 

Now assume that $X$ satisfies   Condition $(L_2)$.
Set $\Sigma_i := \sum_{j=1}^{i}\alpha_j$, for $i\in [m]$, and notice that it holds $\Sigma_i \ge 0$, by assumption. 
Using summation by parts, we have  
\[
\sum_{i=1}^{m} \Delta_i\cdot \alpha_i = 
\Delta_m \cdot \Sigma_m + 
\sum_{i=1}^{m-1} (\Delta_{i}-\Delta_{i+1})\cdot \Sigma_i  \ge 0 \, .
\]
Hence~\eqref{positive} holds, and the result follows. 
\end{proof}

\section{Poisson random variables}\label{sec:4}

In this section we apply Theorem~\ref{gen_thm} to sums of independent Poisson random variables. We obtain the following. 

\begin{theorem}\label{Kane_thm2}
Let $\{\lambda_i\}_{i\ge 1}$ be a sequence of positive integers satisfying $\sum_{i=1}^k\lambda_i \ge \lambda_{k+1}$, for all $k\ge 1$, and let $\{X_i\}_{i\ge 1}$ be 
independent random variables such that 
$X_i\sim \Poi(\lambda_i)$. Then it holds 
\[
\Prb\left( \sum_{i=1}^kX_i\ge \sum_{i=1}^{k} \lambda_i \right) \ge \Prb\left( \sum_{i=1}^{k+1}X_i\ge\sum_{i=1}^{k+1} \lambda_i \right), \, \text{ for } \, k\ge 1 \, .
\]
\end{theorem}

Let us remark that a special case of Theorem~\ref{Kane_thm2}, i.e., when all $\lambda_i$ are equal to a given positive integer $\lambda$, has been obtained  in~\cite[Theorem~2.3]{Kane}, via an analytic approach.  
We break the proof of Theorem~\ref{Kane_thm2} in several lemmata. 

\begin{lemma}
\label{poi_simmons}
Fix a positive integer $m$, and let $X\sim\Poi(m)$.
Then  
\[
\Prb(X\le m-1) > \Prb(X\ge m+1) \, .
\]
\end{lemma}
\begin{proof}
See~\cite[Proposition~3.3]{Perrin_Redside}. 
\end{proof}

\begin{lemma}\label{lem:poiskew}
Fix a positive integer $s$, and let $S\sim\Poi(s)$. 
Then $S$ is right-skewed.  
\end{lemma}
\begin{proof}
Since $s$ is a positive integer it follows that the mode of $S$ is equal to $s$. 
For $i\in [s]$, let 
$\beta_i = \frac{\Prb(S=s-i)}{\Prb(S=s+i-1)}$. 
Since the mode of $S$ is equal to $s$, it follows that $\beta_1\le 1$. Now note that $\beta_i\ge \beta_{i+1}$ is equivalent to $s^2 \ge s^2 - i^2$, 
which is clearly correct for each $i\in [s]$. Hence $\{\beta_i\}_{i=1}^s$ 
is non-increasing, and the fact that $\beta_1\le 1$ finishes the proof. 
\end{proof}

The following result is presumably reported somewhere in the literature but, lacking a reference, we include a proof for the sake of completeness. 

\begin{lemma}\label{poi_sig}
Fix a positive integer $m\ge 3$, and let $X\sim\Poi(m)$. Then it holds 
\[
\Prb(X\ge 2m) \ge \Prb(X=0) \, . 
\]
\end{lemma}
\begin{proof}
It is enough to show that 
$\Prb(X=2m) \ge \Prb(X=0)$ or, equivalently, that $m^{2m} \ge (2m)!\,$. The inequality is clearly correct when $m=3$, so we may suppose that $m\ge 4$. Now consider the function $f(x)=x(2m-x)$, for $x\in (0,2m)$, and notice that is attains its maximum at $x=m$. Hence it holds 
$f(i) \le m^2$. Moreover, when $m\ge 4$ it also holds $2\cdot (2m-1)\le m^2$. Putting those observations together, we deduce 
\[
(2m)! =m^2\cdot 2(2m-1)\cdot \prod_{i=2}^{m-1} \big(i\cdot (2m-i)\big) \le m^2 \cdot m^2\cdot (m^2)^{m-2} = m^{2m} \, ,
\]
as desired. 
\end{proof}

A sequence of real numbers $\{a_i\}_{i=1}^m$, is said to be  \emph{U-shaped} if there exists $\ell\in [m]$ such that $a_1\ge \cdots \ge a_{\ell}$ and $a_{\ell} \le \cdots \le a_m$. 

\begin{lemma}\label{lem:pois}
Let $m\ge 3$ be an integer, and let  $X\sim\Poi(m)$. Then $X$ is left-loaded.  
\end{lemma}
\begin{proof}
We show that $X$ satisfies Condition $(L_1)$. 
Recall that $\alpha_i=\Prb(X\le m-i)-\Prb(X\ge m+i)$. We have to show that $\{\alpha_i\}_{i=1}^{m}$ changes sign once. 
Lemma~\ref{poi_simmons} implies that $\alpha_1 >0$ and Lemma~\ref{poi_sig} implies that $\alpha_m\le 0$, and it thus enough to show that the sequence $\{\alpha_i\}_{i=1}^{m}$ is U-shaped. 
Since, for every $i\in [m-1]$, it holds 
\begin{eqnarray*}
\alpha_{i+1} &=& \alpha_{i} - \Prb(X=m-i) + \Prb(X=m+i)  \, ,
\end{eqnarray*} 
it follows that it is enough to show that the 
sequence $\{a_i\}_{i=1}^{m}$, where $a_i:= \Prb(X=m-i) - \Prb(X=m+i)$, changes sign once. 
To this end, 
for $i\in [m]$, let $\beta_i = \frac{\Prb(X=m+i)}{\Prb(X=m-i)}$. 
Then $\beta_i\ge \beta_{i+1}$ is equivalent to $i^2+i\le m$. Since the sequence 
$\{i^2 + i\}_{i=1}^m$ is increasing, 
it follows that the sequence $\{\beta_i\}_{i=1}^{m}$ is U-shaped. 
Now note that $\beta_1 < 1$, and that the proof of  Lemma~\ref{poi_sig} implies that $\beta_m\ge 1$
Since $\{\beta_i\}_{i=1}^{m}$ is a U-shaped sequence, 
it follows that there exists a unique  
$k\in [m]$ such that 
$\beta_i< 1$, for $i\le k$, and 
$\beta_i \ge 1$, for $i\ge k+1$, which in turn yields that $a_i>0$ for
$i\le k$, and 
$a_i \le 0$, for $i\ge k+1$. 
In other words, the sequence $\{a_i\}_{i=1}^{m}$ changes sign once, as desired. 
\end{proof}

The proof of Theorem~\ref{Kane_thm2} is almost complete. 

\begin{proof}[Proof of Theorem~\ref{Kane_thm2}] 
Let $S\sim \sum_{i=1}^{k}X_i$ and $X\sim X_{k+1}$, and note that $S\sim \Poi(\sum_{i=1}^{k}\lambda_i)$. Set 
$s = \sum_{i=1}^{k}\lambda_i$ and $m=\lambda_{k+1}$. Then $S\sim \Poi(s), X\sim\Poi(m)$ and $s\ge m$, and we proceed with verifying the two conditions in Theorem~\ref{gen_thm}. The mode of $S$ is equal to $s$, and Lemma~\ref{lem:poiskew} implies that $S$ is right-skewed; hence it only remains to verify that $X$ is left-loaded. 
If $m=1$, then the second statement in Lemma~\ref{lem:1} implies that $X$ satisfies Condition $(L_2)$. If $m=2$, then Lemma~\ref{poi_simmons} and the second statement 
in Lemma~\ref{lem:1} imply that $X$ satisfies    
Condition $(L_2)$. If $m\ge 3$ then Lemma~\ref{lem:pois} implies that $X$ satisfies   Condition $(L_1)$. 
The result follows. 
\end{proof}

\section{Concluding remarks}\label{sec:last}

Notice that in the proof of Theorem~\ref{Kane_thm2}
we embarked on the approach of showing that 
the random variable $X\sim\Poi(m)$ satisfies 
Condition $(L_2)$, when $m\in\{1,2\}$, and Condition $(L_1)$, when $m\ge 3$. 
Similar ideas can be applied to sums of 
independent binomial random variables. 
In the case of binomial random variables, 
the following generalisation of~\eqref{CB_ineq} holds true.

\begin{theorem}[Jogdeo -- Samuels~\cite{Jogdeo_Samuels}]
\label{Jog_Sam}
Fix positive integers $n\ge m$, and let $\{X_i\}_{i\ge 1}$ be independent random variables such that   $X_i\sim\Bin(n, \frac{m}{n})$.  
Then it holds 
\begin{equation}\label{jog.sam.ineq}
\Prb\left( \sum_{i=1}^kX_i\ge km \right) \ge \Prb\left( \sum_{i=1}^{k+1}X_i\ge (k+1)m \right), \, \text{ for } \, k\ge 1 \, .
\end{equation}
\end{theorem}

An analytic proof of Theorem~\ref{Jog_Sam} can be found in~\cite[Theorem~2.1]{Kane}. Let us remark that 
Theorem~\ref{gen_thm} yields a proof of a special case of Theorem~\ref{Jog_Sam}. 
Indeed, all lemmata in Section~\ref{sec:4} 
have a corresponding analogue for binomial 
random variables. 
It is easy to see that $S\sim\Bin(nk,\frac{m}{n})$ is left-skewed, when $n\ge 2m$. Moreover, 
$X\sim\Bin(n,\frac{m}{n})$ satisfies Condition $(L_2)$, when $m\in \{1,2\}$, and, using 
arguments similar to the ones in Lemmata~\ref{poi_sig} and~\ref{lem:pois}, 
it can be shown that $X$ satisfies Condition $(S_1)$, when $4\le m\le \frac{n}{3}$. 
In other words, Theorem~\ref{gen_thm} yields 
a proof of~\eqref{jog.sam.ineq} when $m\in\{1,2\}$ or 
$4\le m\le \frac{n}{3}$. 
Clearly, this range of the parameters $n$ and $m$   is 
rather small when compared to the corresponding range in Theorem~\ref{Jog_Sam}, which is valid for $n\ge m$. An application of Theorem~\ref{gen_thm} could increase this range to $n>2m$, provided  $X$ satisfies Condition $(S_2)$. We are unable to verify this 
condition, but we believe that it holds true. 

\begin{conjecture}\label{conj:1}
Fix positive integers $n,m$ such that $n> 2m$, and 
let  $X\sim\Bin(n,\frac{m}{n})$. Then $X$ satisfies   Condition $(S_2)$.
\end{conjecture}

Let us remark that Conjecture~\ref{conj:1}, if true, may be seen as a refinement of an old result due to Simmons (see~\cite{Perrin_Redside, Simmons, Hoogstrate}).

Similarly, the following statement, provided correct, combined with Theorem~\ref{gen_thm} may be employed to establish yet another proof of Theorem~\ref{Kane_thm2}. 

\begin{conjecture}
Fix a positive integer $m$, and let $X\sim\Poi(m)$. 
Then $X$ satisfies 
Condition $(S_2)$. 
\end{conjecture}

\textbf{Acknowledgements } 
S.~P. and C.~P. were supported by the Hellenic Foundation for 
Research and Innovation (H.F.R.I.) under the ``First Call for H.F.R.I. 
Research Projects to support Faculty members and 
Researchers and the procurement of high-cost 
research equipment grant" (Project Number: HFRI-FM17-2436).

\end{document}